\documentclass{amsart}
\usepackage{amsmath,amssymb}
\usepackage{bm}
\usepackage{graphicx}
\usepackage{float}
\usepackage{ascmac}
\usepackage{fancybox}
\usepackage{tabularx}
\usepackage{multicol}

\newtheorem{dfn}{Definition}[section]
\newtheorem{thm}[dfn]{Theorem}
\newtheorem{prop}[dfn]{Proposition}
\newtheorem{lem}[dfn]{Lemma}

\newtheorem{remark}[dfn]{Remark}
\newtheorem{example}[dfn]{Example}

\makeatletter
\renewcommand{\p@enumi}{A.}
\makeatother

\numberwithin{equation}{section}

\title[Inverse problems for first-order hyperbolic equations]{Inverse problems for first-order hyperbolic equations with time-dependent coefficients}
\author{Giuseppe Floridia}
\address{Università Mediterranea di Reggio Calabria,
Department PAU
Via dell'Università 25  
89124 Reggio Calabria, Italy}
\email{floridia.giuseppe@icloud.com}

\author{Hiroshi Takase}
\address{Graduate School of Mathematical Sciences,
The University of Tokyo, 3-8-1 Komaba, Meguro-ku, Tokyo 153-8914, Japan}
\email{htakase@ms.u-tokyo.ac.jp}
\date{August 23, 2021.}
\begin{document}
\begin{abstract}
We prove global Lipschitz stability for inverse source and coefficient problems for first-order linear hyperbolic equations, the coefficients of which depend on both space and time. We use a global Carleman estimate, and a crucial point, introduced in this paper, is the choice of the length of integral curves of a vector field generated by the principal part of the hyperbolic operator to construct a weight function for the Carleman estimate. These integral curves correspond to the characteristic curves in some cases.
\end{abstract}
\maketitle
\section{Introduction}\label{Introduction}
Let $d\in\mathbb{N}$, $\Omega\subset\mathbb{R}^d$ be a bounded domain with Lipschitz boundary $\partial\Omega$, $T>0$, and $Q:=\Omega\times(0,T)$. For $a,b\in\mathbb{R}^d$, we denote by $a\cdot b$ the inner product on $\mathbb{R}^d$. We define the first-order partial differential operator $P$ such that
\[Pu:=A^0(x,t)\partial_tu+A(x,t)\cdot\nabla u,\]
where $A^0\in C^1(\overline{Q})\cap L^\infty(\Omega\times(0,\infty))$ is a positive function, and $A=(A^1,\cdots,A^d)\in C^2(\overline{Q};\mathbb{R}^d)$ is a vector-valued function on $\overline{Q}$. In this paper, we obtain global Lipschitz stability results for three inverse problems for equations with the principal part of type $P$.

\subsubsection*{State of the art} The arguments of this paper are based on the Carleman estimates, which were introduced by Carleman in \cite{Carleman1939} to prove unique continuation properties for elliptic partial differential equations with not necessarily analytic coefficients, and the Bukhgeim--Klibanov method introduced in \cite{Bukhgeim1981}. The methodology using the Carleman estimates is widely applicable to not only inverse problems and unique continuation (e.g., \cite{Yamamoto2017}, \cite{Huang2020}, \cite{Isakov2017}, \cite{Katchalov2001}, \cite{Klibanov2013}, \cite{Lavrentiev1980}, and \cite{Yamamoto2009}), but also control theory (e.g., \cite{Cannarsa2016}, \cite{Chaves-Silva2014}, \cite{Fu2019}, \cite{Fursikov1996}, \cite{Lebeau1995}, and \cite{Porretta2012}) for various partial differential equations.

Now, we describe some results concerned with the operator $P$. For the radiative transport equation having the principal part of type
\[\partial_tu(x,v,t)+v\cdot \nabla u(x,v,t),\quad (x,v,t)\in \Omega\times \mathbb{S}^{d-1}\times(0,T),\]
where $\mathbb{S}^{d-1}:=\{v\in\mathbb{R}^d\mid |v|=1\}$ is a set of a velocity field, Klibanov and Pamyatnykh \cite{Klibanov2006} and \cite{Klibanov2008} proved the Carleman estimates and global uniqueness theorem for inverse coefficient problem of determining a zeroth-order coefficient. In \cite{Klibanov2006} and \cite{Klibanov2008}, the weight function for the Carleman estimate was independent of the principal parts:
\[\varphi(x,t)=|x-x_0|^2-\beta t^2,\]
where $x_0\in\mathbb{R}^d$ and $\beta>0$ were fixed. For the same weight function used for transport equations with space-dependent first-order coefficients, see also Gaitan and Ouzzane \cite{Gaitan2014}. Machida and Yamamoto \cite{Machida2014} and \cite{Machida2020} also proved global Lipschitz stability for inverse coefficient problems, where they took a linear function as the weight function for the Carleman estimate:
\[\varphi(x,t)=\gamma\cdot x-\beta t,\]
where $\gamma\in\mathbb{R}^d$ and $\beta>0$ were fixed. Recently, Lai and Li \cite{Lai2020} proved Lipschitz stability for inverse source and coefficient problems of determining a zeroth-order coefficient under the assumption that there existed a suitable weight function for the Carleman estimate.

For first-order hyperbolic operators of type $P$ with a variable principal part, G\"olgeleyen and Yamamoto \cite{Golgeleyen2016} proved Lipschitz stability and conditional H\"older stability for inverse source and inverse coefficient problems, where they assumed the existence of a suitable weight function $\varphi=\varphi(x,t)$ for the Carleman estimate satisfying
\[\min_{(x,t)\in\overline{Q}}P\varphi(x,t)>0\]
when $A^0\equiv1$ and $A=A(x)$. In the same time-independent case, Cannarsa, Floridia, G\"olgeleyen, and Yamamoto \cite{Cannarsa2019} proved local H\"older stability for inverse coefficient problems of determining the principal part and a zeroth-order coefficient, where they took a function
\[\varphi(x,t)=A(x)\cdot x-\beta t\]
as the weight function for the Carleman estimate, and determined the coefficients up to a local domain, depending on the weight function, from local boundary data. In the same time-independent case, we also mention that Gaitan and Ouzzane \cite{Gaitan2014} proved global Lipschitz stability for inverse coefficient problem of determining a zeroth-order coefficient via the Carleman estimate.

In these results mentioned above, in general, one must impose some assumptions on the principal parts and weight functions to guarantee the Carleman estimates that is not needed in this paper. Moreover, we must note that these results were all for first-order equations with coefficients independent of time $t$. However, equations with time-dependent principal parts of type $P$ often appear in mathematical physics, for example, the conservation law of mass in time-dependent velocity fields, and the mathematical analysis for such equations is needed (e.g., Taylor \cite[Section 17.1]{Taylor2011III} and Evans \cite[Section 11.1]{Evans2010}). In regard to first-order hyperbolic equations having time-dependent principal parts, although the theory about direct problems for the above equations is quite complete, there are some open questions for inverse problems due to the major difficulties in dealing with time-dependent coefficients. About inverse problems and time-dependent principal parts, we mention Cannarsa, Floridia, and Yamamoto \cite{Cannarsa2019a} that proved an observability inequality for a non-degenerate case. Floridia and Takase \cite{Floridia2020a} proved the observability inequality for a degenerate case, which was motivated by applications to inverse problems. In both papers, they dealt with the case $A^0\equiv1$ and $A=A(t)$. For more references regarding inverse problems and controllability for conservation laws with time-dependent coefficients, see \cite{Holden2014}, \cite{Kabanikhin2020}, \cite{Kang2005}, and \cite{Klibanov2007}. Regarding inverse problems for nonlinear first-order equations, readers are referred to Esteve and Zuazua \cite{Esteve2020}, which studies Hamilton--Jacobi equations (see also Porretta and Zuazua \cite{Porretta2012}).

For the second-order hyperbolic equations with time-dependent coefficients, the literature about inverse problems is more extensive. In this context, Jiang, Liu, and Yamamoto \cite{Liu2017}, and Yu, Liu, and Yamamoto \cite{Yu2018} proved the local H\"older stability for inverse source and coefficient problems in the Euclidean space assuming the Carleman estimates existed. Takase \cite{Takase2020a} proved local H\"older stability for the wave equation and obtained  some sufficient conditions for the Carleman estimate by using geometric analysis on Lorentzian manifolds.

Finally, we note that, on the well-posedness by the method of characteristics of first-order hyperbolic equations with principal parts of type $P$, readers are referred to John \cite[Chapter 1]{John1978}, Rauch \cite[Chapter 1]{Rauch2012}, Evans \cite[Chapter 3]{Evans2010}, and Bressan \cite{Bressan2000}. In addition to that, for symmetric hyperbolic systems, readers are referred to Rauch \cite[Chapter 2]{Rauch2012}, Ringstr\"om \cite[Chapter 7]{Ringstrom2009}, and Taylor \cite[Section 16.2]{Taylor2011III}.

\subsubsection*{Purpose of this paper} Although a large number of studies have been made on inverse problems for first-order equations, as already mentioned, what seems to be lacking is analysis for equations with time-dependent coefficients. In this paper we investigate equations with coefficients depending on both space and time. The important point we want to make is the decisive way to choose the weight function in the Carleman estimate for applications to inverse problems. Indeed, the weight function of our Carleman estimate (see Proposition \ref{Carleman} and Lemma \ref{lemma}) is linear in $t$, which is similar to Machida--Yamamoto \cite{Machida2014}, G\"olgeleyen--Yamamoto \cite{Golgeleyen2016}, and Cannarsa--Floridia--Yamamoto \cite{Cannarsa2019a}. However, the novelty is that the spatial term of the weight function in our Carleman estimate is the length of integral curves of the vector-valued function $A(\cdot,0)$, which is different from the ones in all the above results (\cite{Cannarsa2019a}, \cite{Floridia2020a}, \cite{Gaitan2014}, \cite{Golgeleyen2016}, \cite{Klibanov2006}, \cite{Klibanov2008}, and \cite{Machida2014}) and a new attempt. Owing to the choice, we need not assume any assumptions on $A$ to guarantee the Carleman estimates like in \cite{Golgeleyen2016} and \cite{Cannarsa2019}, but assume only the finiteness of the length of integral curves (see Definition \ref{dissipative} and \eqref{finiteness}). We remark that these integral curves correspond to the characteristic curves in the case $A^0\equiv 1$ and $A=A(x)$. In addition, we note that thanks to the above linearity with respect to $t$, we do not need to extend the solution to $(-T,0)$, which enables us to apply the Carleman estimate to inverse problems for wider functional space of time-dependent coefficients $A^0$ and $A$.

\subsubsection*{Structure of this paper} The main results in this paper are global Lipschitz stability for the inverse source problem (Theorem \ref{ISP}), inverse coefficient problem to determine the zeroth-order coefficient (Theorem \ref{ICP}), and inverse coefficient problem to determine the time-independent principal part (Theorem \ref{ICP2}). After describing some settings, we present them in section \ref{Preliminary}. In section \ref{Carleman_energy}, we establish the global Carleman estimate (Proposition \ref{Carleman}), which is the main tool to prove the main results, under the assumption that a suitable weight function exists. After that, we prove the existence of such a weight function by taking the length of integral curves generated by the vector-valued function $A(\cdot,0)$ (Lemma \ref{lemma}). In addition, in section \ref{Carleman_energy}, we introduce energy estimates needed to prove the main results. In section \ref{Proof}, we show the proofs of the main results. In Appendix, we give the proofs of auxiliary and original results.

\section{Preliminary and statements of main results}\label{Preliminary}
Before showing main results, we describe some definitions and settings needed to present them.

\begin{dfn}
For a vector-valued function $X\in C^2(\overline{\Omega};\mathbb{R}^d)$ and $x\in\overline{\Omega}$, a $C^2$ curve $c:[-\eta_1,\eta_2]\to \overline{\Omega}$ for some $\eta_1\ge 0$ and $\eta_2\ge 0$ with $\eta_1+\eta_2>0$ is called an integral curve of $X$ through $x$ if it solves the following initial problem for ordinary differential equations
\[\begin{cases}\displaystyle c'(\sigma):=\frac{dc}{d\sigma}(\sigma)=X(c(\sigma)),\quad \sigma\in[-\eta_1,\eta_2],\\ c(0)=x.\end{cases}\]
\end{dfn}

\begin{remark}\label{extension}
If $c_x$ denotes the integral curve of $X$ through $x$, then $c_x(\sigma)$ is $C^2$ with respect to $x\in\overline{\Omega}$.
\end{remark}

\begin{dfn}
Let $a,b\in\mathbb{R}$ with $a<b$. An integral curve $c:[a,b]\to\overline{\Omega}$ is called maximal if it cannot be extended in $\overline{\Omega}$ to a segment $[a-\eta_1,b+\eta_2]$ for some $\eta_1\ge 0$ and $\eta_2\ge 0$ with $\eta_1+\eta_2>0$.
\end{dfn}

\begin{dfn}\label{dissipative}
A vector-valued function $X\in C^2(\overline{\Omega};\mathbb{R}^d)$ is called dissipative if the maximal integral curve $c_x$ of $X$ through $x$ is defined on a finite segment $[\sigma_-(x),\sigma_+(x)]$ and $\sigma_-\in C(\overline{\Omega})\cap H^2(\Omega)$.
\end{dfn}

\begin{remark}
If $X\in C^2(\overline{\Omega};\mathbb{R}^d)$ is dissipative, then $c_x(\sigma_-(x))$, $c_x(\sigma_+(x))\in\partial\Omega$, where $c_x$ is the maximal integral curve of $X$ through $x$.
\end{remark}

\begin{figure}[htbp]
\centering\includegraphics[scale=0.25]{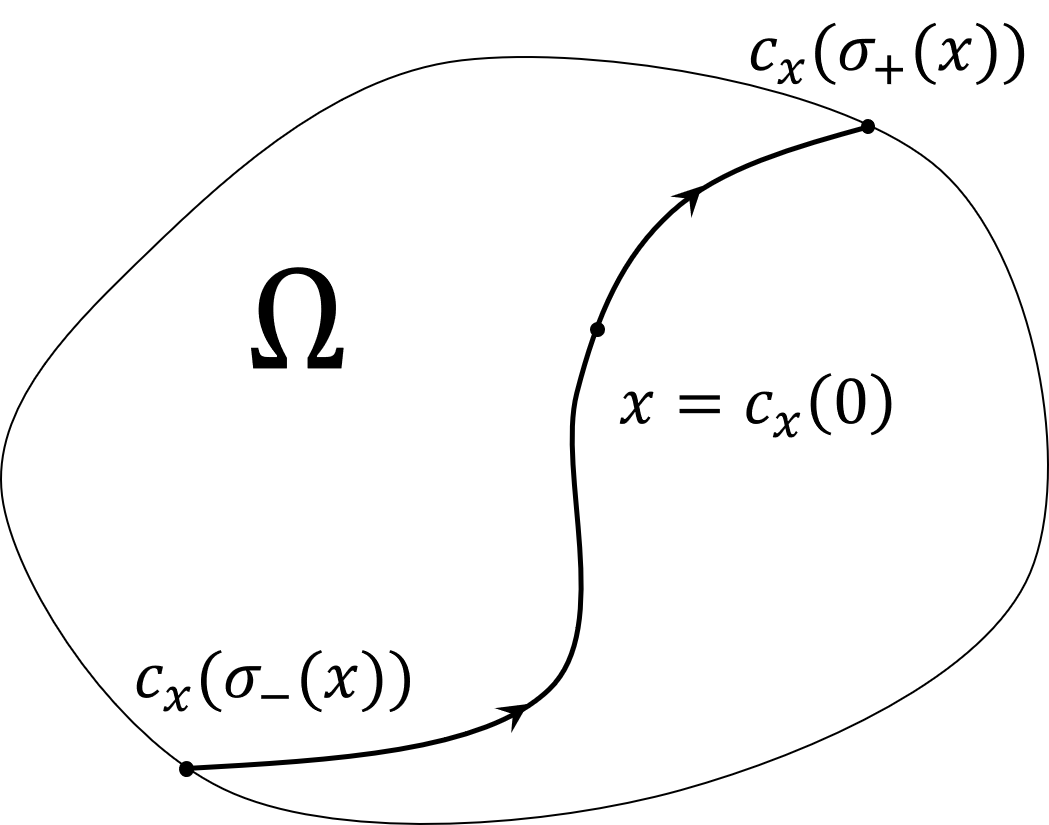}
\caption{$c_x$ is the maximal integral curve of $X$ through $x$.}
\end{figure}

The terminology dissipative for vector fields seems not to be widely-used. However, the authors use this terminology on the analogy of CDRM (compact dissipative Riemannian manifold) used in a setting of integral geometry problems for tensor fields. In this subject, CDRM is equivalent to the absence of a geodesic of infinite length in a compact Riemannian manifold with strictly convex boundary (e.g., \cite[Chapter 4]{Sharafutdinov1994}).

We assume the followings on the vector-valued function $A\in C^2(\overline{Q};\mathbb{R}^d)$:

\begin{equation}\label{positivity}\exists\rho>0\ \text{s.t.}\ \min_{(x,t)\in\overline{Q}}|A(x,t)|\ge\rho\ ;\end{equation}
\[\exists t_*\in[0,T)\ \text{s.t.}\ A(\cdot,t_*)\ \text{is dissipative}.\]
Without loss of generality, we assume $t_*=0$ in the above, i.e.,
\begin{equation}\label{finiteness}A(\cdot,0)\ \text{is dissipative}\end{equation}
because it suffices to consider the change of variables $\tilde{t}:=t-t_*$ and $\tilde{A}(\cdot,\tilde{t}):=A(\cdot,\tilde{t}+t_*)$.

\begin{remark}
In the case $A^0\equiv 1$ and $A=A(x)$, \eqref{finiteness} means that any maximal characteristic curves have finite length.
\end{remark}

\begin{example}
Let $d=2$ and $B_r:=\{(x,y)\in \mathbb{R}^2\mid x^2+y^2<r^2\}$ for $r>0$. Then, $X(x,y):=\begin{pmatrix}-x\\ -1\end{pmatrix}$ on $\Omega=B_r\cap \{y>0\}$ is dissipative because we see $\sigma_-$ is smooth on $\overline{\Omega}$. However, $Y(x,y):=\begin{pmatrix}-y\\ x\end{pmatrix}$ on $B_r\setminus \overline{B_\frac{r}{2}}$ is not dissipative because we can not define $\sigma_-$.

\begin{figure}[htbp]
\centering\includegraphics[scale=0.4]{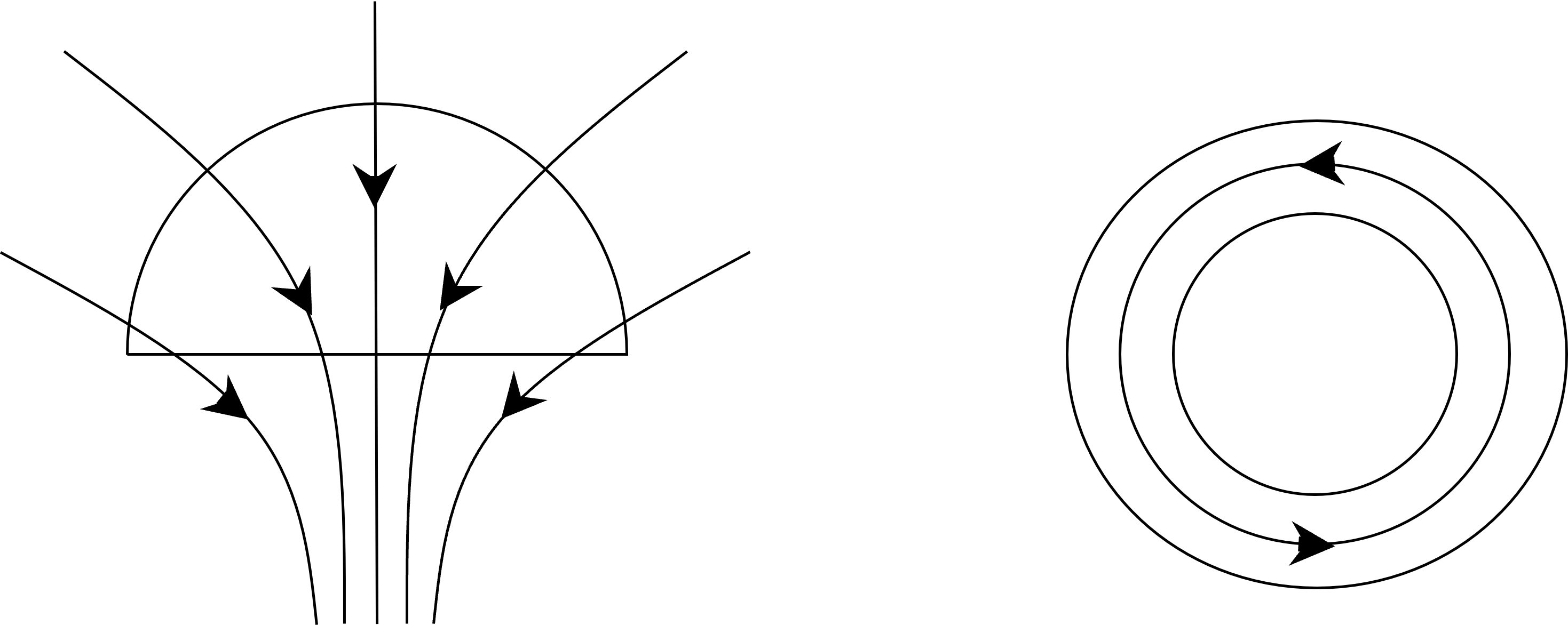}
\caption{Pictures of $X$ (left) and $Y$ (right).}
\end{figure}
\end{example}

Under the assumption \eqref{finiteness}, we can give the following notations. For a fixed $x\in\overline{\Omega}$, let $c_x:[\sigma_-(x),\sigma_+(x)]\to\overline{\Omega}$ be the maximal integral curve of $A(\cdot,0)$ through $x$, i.e., $c_x$ satisfies
\[\begin{cases}c_x'(\sigma)=A(c_x(\sigma),0),\quad \sigma\in[\sigma_-(x),\sigma_+(x)],\\ c_x(0)=x.\end{cases}\]
Since $c_x$ is a rectifiable curve by \eqref{finiteness}, we can define the function $\varphi_0$ on $\overline{\Omega}$ as the length of the arc of the maximal integral curves defined on $[\sigma_-(x),0]$:
\begin{equation}\label{distance}\varphi_0(x):=\int_{\sigma_-(x)}^0|c_x'(\sigma)|d\sigma,\end{equation}
the integral of which is independent of a choice of parameters.

\begin{lem}\label{regularity}
Let $A\in C^2(\overline{Q};\mathbb{R}^d)$ be a vector-valued function. Assume \eqref{positivity} and \eqref{finiteness}. Then, the function $\varphi_0$ defined by \eqref{distance} is in the class $C(\overline{\Omega})\cap H^2(\Omega)$.
\end{lem}
\begin{proof}It follows from Definition \ref{dissipative} and Remark \ref{extension}.\end{proof}

To prove the global Lipschitz stability for inverse problems for the hyperbolic equations, the observation time should be given large enough for the solutions to reach the boundaries owing to the finite propagation speeds (see Bardos, Lebeau, and Rauch \cite{Bardos1992}). Then, we define the following quantities to describe this situation mathematically.

For the positive function $A^0\in C^1(\overline{\Omega})\cap L^\infty(\Omega\times(0,\infty))$ and $\varphi_0$ defined by \eqref{distance}, we define the positive number
\begin{equation}\label{T_0}T_0:=\frac{\displaystyle \Big(\sup_{x\in\Omega,t>0}A^0(x,t)\Big)\Big(\max_{x\in\overline{\Omega}}\varphi_0(x)\Big)}{\rho}.\end{equation}

Moreover, considering inverse problems for the hyperbolic equation with time-dependent principal part, we will assume
\begin{equation}\label{spd}\exists C>0\ \text{s.t.}\ \forall\xi\in\mathbb{R}^d,\ \forall(x,t)\in\overline{Q},\quad |\partial_tA(x,t)\cdot\xi|\le C|A(x,t)\cdot\xi|.\end{equation}
The condition \eqref{spd} will be decisive in the the energy estimate given in Lemma \ref{energy} and in the proofs of Theorem \ref{ISP} and Theorem \ref{ICP}. 

\begin{remark}\label{d=1}
When $d=1$, \eqref{positivity} implies \eqref{spd}.
\end{remark}

If a non-vanishing vector valued function $A$ satisfies \eqref{spd}, then $A$ has the following structure.

\begin{prop}\label{structure}
If a vector-valued function $A\in C^2(\overline{Q};\mathbb{R}^d)$ satisfies \eqref{positivity} and \eqref{spd}, then $A$ can be represented by
\[A(x,t)=A(x,0)e^{\int_0^t\phi(x,s)ds},\quad (x,t)\in \overline{Q}\]
for some function $\phi\in C^1(\overline{Q})$.
\end{prop}
The proof of Proposition \ref{structure} is presented in Appendix. Proposition \ref{structure} is decisive in the realization of a weight function for the Carleman estimate, which will be given in Lemma \ref{lemma}.

Now, we define some notations. Set
\[\Sigma_+:=\{(x,t)\in\partial\Omega\times(0,T)\mid A(x,t)\cdot\nu(x)>0\},\]
where we recall $\nu$ is the outer unit normal to $\partial\Omega$. Moreover, we set $\Sigma_-:=(\Sigma_+)^c=(\partial\Omega\times(0,T))\setminus\Sigma_+$.

We use the notations $H^0(\Omega):=L^2(\Omega)$, $H^0(0,T;H^1(\Omega)):=L^2(0,T;H^1(\Omega))$, and $\partial_t^0w=w$ for a function $w$ throughout this paper to avoid notational complexity.
\subsection{Inverse source problems}
We consider the initial boundary value problem
\begin{align}\label{boundary}\begin{cases}Pu+p(x,t)u=R(x,t)f(x)\quad &\text{in}\ Q,\\
u=0\quad &\text{on}\ \Sigma_-,\\
u(\cdot,0)=0\quad &\text{on}\ \Omega,\end{cases}\end{align}
where $p\in W^{1,\infty}(0,T;L^\infty(\Omega))$, $R\in H^1(0,T;L^\infty(\Omega))$, and $f\in L^2(\Omega)$. Given $A^0$, $A$, $p$, and $R$, we consider the inverse source problem to determine the source term $f$ in $\Omega$ by observation data $u$ on $\Sigma_+$.

\begin{thm}\label{ISP}
Let $A^0\in C^1(\overline{Q})\cap L^\infty(\Omega\times(0,\infty))$ satisfying $\displaystyle\min_{(x,t)\in\overline{Q}}A^0(x,t)>0$, and $A\in C^2(\overline{Q};\mathbb{R}^d)$ satisfying \eqref{positivity}, \eqref{finiteness}, and \eqref{spd}. Let $p\in W^{1,\infty}(0,T;L^\infty(\Omega))$, $R\in H^1(0,T;L^\infty(\Omega))$, and $f\in L^2(\Omega)$ satisfying
\begin{equation}\label{R}\exists m_0>0\ \text{s.t.}\ |R(x,0)|\ge m_0\quad a.e.\ x\in\Omega.\end{equation}
Assume
\begin{equation}\label{time}T_0<T,\end{equation}
where $T_0$ is defined by \eqref{T_0}, and there exists a function $u$ satisfying \eqref{boundary} in the class
\[u\in \bigcap_{k=1}^2H^{k}(0,T;H^{2-k}(\Omega)).\]
Then, there exists a constant $C>0$ independent of $f$ and $u$ such that
\[\|f\|_{L^2(\Omega)}\le C\sum_{k=0}^1\|\partial_t^ku\|_{L^2(\Sigma_+)}.\]
\end{thm}

\subsection{Inverse coefficient problems}
We consider the initial boundary value problem
\begin{align}\label{boundary2}\begin{cases}Pu+p(x,t)u=0\quad &\text{in}\ Q,\\
u=g\quad &\text{on}\ \Sigma_-,\\
u(\cdot,0)=\alpha\quad &\text{on}\ \Omega,\end{cases}\end{align}
where $p\in W^{1,\infty}(0,T;L^\infty(\Omega))$, $\displaystyle g\in \bigcap_{k=1}^2H^k(0,T;H^{\frac{3}{2}-k}(\partial\Omega))$, and $\alpha\in H^1(\Omega)$ satisfying the compatibility conditions. In the following two subsections, we present two nonlinear inverse coefficient problems.

\subsubsection{Zeroth-order coefficient}\label{Zeroth-order}
Given $A^0$, $A$, $g$, and $\alpha$, we consider the inverse coefficient problem to determine the time-independent zeroth-order coefficient $p=p(x)$ in $\Omega$ by observation data on $\Sigma_+$.

For a fixed $M>0$, define the conditional set
\[D(M):=\{p\in L^\infty(\Omega)\mid \|p\|_{L^\infty(\Omega)}\le M\}.\]

\begin{thm}\label{ICP}
Let $M>0$ be fixed, $A^0\in C^1(\overline{Q})\cap L^\infty(\Omega\times(0,\infty))$satisfying $\displaystyle\min_{(x,t)\in\overline{Q}}A^0(x,t)>0$, and $A\in C^2(\overline{Q};\mathbb{R}^d)$ satisfying \eqref{positivity}, \eqref{finiteness}, and \eqref{spd}. Let $p_i\in D(M)$ for $i=1,2$, $\displaystyle g\in \bigcap_{k=1}^2H^k(0,T;H^{\frac{3}{2}-k}(\partial\Omega))$, and $\alpha\in H^1(\Omega)$ satisfying
\begin{equation}\label{alpha}\exists m_0> 0\ \text{s.t.}\ |\alpha(x)|\ge m_0\quad a.e.\ x\in\Omega.\end{equation}
Assume $T_0<T$, where $T_0$ is defined by \eqref{T_0}, and for $i=1,2$ there exist functions $u_i$ satisfying \eqref{boundary2} with $p=p_i$ in the class
\[u_i\in \bigcap_{k=1}^2H^k(0,T;H^{2-k}(\Omega))\]
such that
\[u_2\in H^1(0,T;L^\infty(\Omega))\ \text{and}\ \|u_2\|_{H^1(0,T;L^\infty(\Omega))}\le M.\]
Then, there exists a constant $C>0$ independent of $p_i\in D(M)$ for $i=1,2$ such that
\[\|p_1-p_2\|_{L^2(\Omega)}\le C\sum_{k=0}^1\|\partial_t^ku_1-\partial_t^ku_2\|_{L^2(\Sigma_+)}.\]
\end{thm}

\subsubsection{First-order coefficients}
We consider \eqref{boundary2} with the time-independent principal coefficients $A^0$ and $A$, more precisely, with $A^0\in C^1(\overline{\Omega})$ and $A\in C^2(\overline{\Omega};\mathbb{R}^d)$. Given $p$, finitely many initial values $\alpha$, and boundary values $g$, we consider the inverse coefficient problem to determine the time-independent coefficients $A^0$ and $A$ simultaneously by finitely many observation data on $\Sigma_+$.

Let $\rho>0$ be fixed. We will assume that the unknown coefficients $A^0$ and $A$ satisfy the following condition:
\begin{equation}\label{time2}\frac{\displaystyle\Big(\max_{x\in\overline{\Omega}}A^0(x)\Big)\Big(\max_{x\in\overline{\Omega}}\varphi_0(x)\Big)}{\rho}<T,\end{equation}
where $\varphi_0$ is defined by \eqref{distance}.

For $A\in C^2(\overline{\Omega};\mathbb{R}^d)$, set
\[\Gamma_{+,A}:=\{x\in\partial\Omega\mid A(x)\cdot\nu(x)>0\}\]
and $\Gamma_{-,A}:=\partial\Omega\setminus\Gamma_{+,A}$.

For fixed $M>0$, $\rho>0$, and a subset $\Gamma\subset\partial\Omega$, define the conditional set
\begin{align*}&D(M,\rho,\Gamma)\\
&:=\left\{(A^0,A)\in C^1(\overline{\Omega})\times C^2(\overline{\Omega};\mathbb{R}^d)\middle| \begin{cases}\|A^0\|_{C^1(\overline{\Omega})}+\|A\|_{C^2(\overline{\Omega};\mathbb{R}^d)}\le M,\\
\displaystyle \min_{x\in\overline{\Omega}}A^0(x)\ge\rho,\ \min_{x\in\overline{\Omega}}|A(x)|\ge\rho,\\
\eqref{finiteness},\ \eqref{time2},\ \text{and}\ \Gamma_{+,A}\subset\Gamma\ \text{hold}.\end{cases}\right\}.\end{align*}

\begin{thm}\label{ICP2}
Let $M>0$, $\rho>0$, $\Gamma\subset\partial\Omega$ be a subset, and $(A^0_i,A_i)\in D(M,\rho,\Gamma)$ for $i=1,2$. Let $p\in W^{1,\infty}(0,T;L^\infty(\Omega))$, $\displaystyle g_m\in \bigcap_{k=1}^2H^k(0,T;H^{\frac{3}{2}-k}(\partial\Omega))$, and $\alpha_m\in W^{1,\infty}(\Omega)$ for $m=1,\ldots,d+1$  satisfying
\begin{equation}\label{R2}\exists m_0>0\ \text{s.t.}\ |p(x,0)|\left|\det \begin{pmatrix}\alpha_1(x)& \cdots&\alpha_{d+1}(x)\\ \nabla\alpha_1(x)& \cdots& \nabla\alpha_{d+1}(x)\end{pmatrix}\right|\ge m_0\quad a.e.\ x\in\Omega.\end{equation}
Assume that for $i=1,2$ and $m=1,\ldots,d+1$ there exist functions $u_{i,m}$ satisfying \eqref{boundary2} with $P=P_i:=A^0_i\partial_t+A_i\cdot\nabla$, $g=g_m$, and $\alpha=\alpha_m$ in the class
\[u_{i,m}\in \bigcap_{k=1}^2H^k(0,T;W^{2-k,\infty}(\Omega))\]
such that for all $m=1,\ldots,d+1$,
\[\sum_{k=1}^2\|u_{2,m}\|_{H^k(0,T;W^{2-k,\infty}(\Omega))}\le M.\]
Then, there exists a constant $C>0$ independent of $(A_i^0,A_i)\in D(M,\rho,\Gamma)$ for $i=1,2$ such that
\[\sum_{\mu=0}^d\|A_1^\mu-A_2^\mu\|_{L^2(\Omega)}\le C\sum_{m=1}^{d+1}\|u_{1,m}-u_{2,m}\|_{H^1(0,T;L^2(\Gamma))}.\]
\end{thm}

\section{Carleman estimate and energy estimates}\label{Carleman_energy}
In this section, we introduce the Carleman estimate and energy estimates needed to prove the main results.
\subsection{Carleman estimate}\label{Carleman_estimate}
In this subsection, we prove the global Carleman estimate for the operator $P+p(x,t)\cdot$, where $p\in L^\infty(Q)$. In section \ref{General_statements}, we present the general statement for the Carleman estimate assuming the existence of a suitable weight function $\varphi$ satisfying some sufficient conditions. In section \ref{Realization}, we construct such a weight function satisfying the sufficient conditions using $\varphi_0$ defined by \eqref{distance}.

\subsubsection{General statements}\label{General_statements}
To obtain the local in time Carleman estimate, we first assume the existence of a function $\varphi\in H^2(Q)$ satisfying
\begin{equation}\label{A.1}\exists\delta>0\ \text{s.t.}\ P\varphi(x,t)\ge\delta\quad \text{a.e.}\ (x,t)\in Q.\end{equation}

\begin{prop}\label{Carleman}
Let $A^0\in C^1(\overline{Q})$ satisfying $\displaystyle\min_{(x,t)\in\overline{Q}}A^0(x,t)>0$, $A\in C^1(\overline{Q};\mathbb{R}^d)$, and $p\in L^\infty(Q)$. Assume that there exists a function $\varphi\in H^2(Q)$ satisfying \eqref{A.1}. Then, there exist constants $s_*>0$ and $C>0$ such that
\begin{align}\label{estimate}
&s^2\int_{Q} e^{2s\varphi}|u|^2dxdt+s\int_\Omega e^{2s\varphi(x,0)}|u(x,0)|^2dx\\
&\le C\int_{Q} e^{2s\varphi}|(P+p(x,t))u|^2dxdt+Cs\int_{\Sigma_+} e^{2s\varphi}|u|^2dS dt\notag\\
&\quad+Cs\int_\Omega e^{2s\varphi(x,T)}|u(x,T)|^2dx\notag\end{align}
holds for all $s>s_*$ and $\displaystyle u\in \bigcap_{k=0}^1H^k(0,T;H^{1-k}(\Omega))$, where $dS$ denotes the area element of $\partial\Omega$.
\end{prop}

\begin{proof}
It suffices to prove Proposition \ref{Carleman} when $p\equiv0$ due to the sufficiently large parameter $s$. Let $z:=e^{s\varphi}u$ and $P_s z:=e^{s\varphi}P(e^{-s\varphi}z)$ for $s>0$. Then, we obtain
\[P_s z=Pz-s P\varphi z,\]
which implies
\begin{align*}&\|P_s z\|_{L^2(Q)}^2=\|Pz\|_{L^2(Q)}^2+2(Pz,-s P\varphi z)_{L^2(Q)}+\|s P\varphi z\|_{L^2(Q)}^2\\
&\ge \|s P\varphi z\|_{L^2(Q)}^2+2(Pz,-s P\varphi z)_{L^2(Q)}\\
&=s^2\int_{Q}|P\varphi|^2|z|^2dxdt-s\int_{Q} P\varphi\Big(A^0\partial_t(|z|^2)+A\cdot\nabla(|z|^2)\Big)dxdt\\
&\ge s^2\int_{Q}\delta^2|z|^2dxdt+s\int_{Q} \Big[\partial_t((P\varphi)A^0)+\nabla\cdot((P\varphi)A)\Big]|z|^2dxdt-\mathcal{B},\end{align*}
by our assumption \eqref{A.1}, where
\[\mathcal{B}:=s\int_\Omega\Big[(P\varphi) A^0|z|^2\Big]_{t=0}^{t=T}dx+s\int_{\partial\Omega\times(0,T)}P\varphi (A(x,t)\cdot\nu(x))|z|^2dSdt.\]
Therefore, there exists $C>0$ such that
\[C\int_{Q} s^2\Big[1+O\left(\frac{1}{s}\right)\Big]|z|^2dxdt\le\|P_s z\|_{L^2(Q)}^2+\mathcal{B}\]
as $s\to+\infty$. By choosing $s>0$ large enough, we complete the proof.
\end{proof}

\subsubsection{Realization of weight functions}\label{Realization}
We construct the weight function $\varphi\in C(\overline{Q})\cap H^2(Q)$ depending on the vector field generated by the coefficients $A$, and satisfying $\eqref{A.1}$.

\begin{lem}\label{lemma}
Let $A^0\in C(\overline{Q})\cap L^\infty(\Omega\times(0,\infty))$ satisfying $\displaystyle\min_{(x,t)\in\overline{Q}}A^0(x,t)>0$, and $A\in C^2(\overline{Q};\mathbb{R}^d)$ be given functions satisfying \eqref{positivity}, \eqref{finiteness}, and \eqref{spd}. Then, for an arbitrary real number $\beta>0$ independent of $T$ satisfying
\begin{equation}\label{beta}0<\beta<\frac{\rho}{\displaystyle\sup_{x\in\Omega,t>0}A^0(x,t)},\end{equation}
the function $\varphi$ defined by
\begin{equation}\label{weight}\varphi(x,t):=\varphi_0(x)-\beta t,\quad (x,t)\in \overline{Q},\end{equation}
with $\varphi_0$ defined by \eqref{distance}, is in the class $\varphi\in C(\overline{Q})\cap H^2(Q)$ and satisfies \eqref{A.1}.
\end{lem}

\begin{proof}
It is obvious that $\varphi\in C(\overline{Q})\cap H^2(Q)$ by Lemma \ref{regularity}. We prove that $\varphi$ defined by \eqref{weight} satisfies \eqref{A.1}. It follows that
\begin{align}\label{min}P\varphi(x,t)&=A(x,t)\cdot\nabla\varphi_0(x)-\beta A^0(x,t)\\
&\ge A(x,t)\cdot\nabla\varphi_0(x)-\beta\displaystyle\sup_{x\in\Omega,t>0}A^0(x,t).\notag\end{align}
For a fixed $x\in\Omega$, let $c_x:[\sigma_-(x),\sigma_+(x)]\to\overline{\Omega}$ be the maximal integral curve with $c_x(0)=x$ of $A(\cdot,0)$. For a sufficiently small $\eta\in[\sigma_-(x),\sigma_+(x)]$, we set $x_\eta:=c_x(\eta)$. Because we can verify
\[\begin{cases}\dfrac{d}{d\sigma}\big(c_x(\sigma+\eta)\big)=c'_x(\sigma+\eta)=A(c_x(\sigma+\eta),0),\\ c_x(0+\eta)=x_\eta,\end{cases}\]
we have $c_{x_\eta}(\sigma)=c_x(\sigma+\eta)$ by the uniqueness of the solution to the initial problem of the ordinary differential equation. Hence, $\sigma_-(x_\eta)=\sigma_-(x)-\eta$ holds. Therefore, we obtain
\[\varphi_0(c_x(\eta))=\varphi_0(x_\eta)=\int_{\sigma_-(x_\eta)}^0|c_{x_\eta}'(\sigma)|d\sigma=\int_{\sigma_-(x)-\eta}^0|c_x'(\sigma+\eta)|d\sigma=\int_{\sigma_-(x)}^\eta|c_x'(\sigma)|d\sigma.\]
Differentiating both sides with respect to $\eta$ and substituting $\eta=0$ yield
\[c_x'(0)\cdot\nabla\varphi_0(c_x(0))=|c_x'(0)|=|A(x,0)|.\]
Therefore, by \eqref{spd}, Proposition \ref{structure}, and \eqref{positivity}, we obtain
\begin{align}\label{Taylor}A(x,t)\cdot\nabla\varphi_0(x)&=A(x,0)\cdot\nabla\varphi_0(x)e^{\int_0^t\phi(x,s)ds}\\
&=c_x'(0)\cdot\nabla\varphi_0(c_x(0))e^{\int_0^t\phi(x,s)ds}\notag\\
&=|A(x,0)|e^{\int_0^t\phi(x,s)ds}\notag\\
&=|A(x,t)|\ge \rho.\notag\end{align}
Applying \eqref{Taylor} to \eqref{min} yields
\[P\varphi(x,t)\ge \rho-\beta\displaystyle\sup_{x\in\Omega,t>0}A^0(x,t)> 0\]
for almost all $(x,t)\in Q$.

\end{proof}
\subsection{Energy estimates}\label{Energy_estimates}
The following Lemma \ref{energy} is the energy estimate for the first-order hyperbolic equations with the time-dependent principal part needed to prove Theorem \ref{ISP} and Theorem \ref{ICP}. Moreover, we describe Lemma \ref{energy2}, which is the energy estimate for first-order hyperbolic equations with time-independent principal part needed to prove Theorem \ref{ICP2}. Their proofs are presented in Appendix.

For a positive function $A^0\in C^1(\overline{Q})$ and $\displaystyle u\in \bigcap_{k=1}^2H^k(0,T;H^{2-k}(\Omega))$, we define the quantity
\[E(t):=\int_\Omega \big(A^0(x,t)|\partial_tu(x,t)|^2+|u(x,t)|^2\Big)dx,\quad t\in[0,T].\]

\begin{lem}\label{energy}
Let $A^0\in C^1(\overline{Q})$ satisfying $\displaystyle\min_{(x,t)\in\overline{Q}}A^0(x,t)>0$, $A\in C^1(\overline{Q};\mathbb{R}^d)$, $p\in W^{1,\infty}(0,T;L^\infty(\Omega))$, $R\in H^1(0,T;L^\infty(\Omega))$, and $f\in L^2(\Omega)$. Then, there exists a constant $C>0$ independent of $u$ and $f$ such that
\begin{equation}\label{useless}E(t)\le C\Big(\|\partial_tA\cdot\nabla u\|_{L^2(Q)}^2+\|f\|_{L^2(\Omega)}^2\Big)\end{equation}
holds for all $t\in[0,T]$ and $\displaystyle u\in \bigcap_{k=1}^2H^k(0,T;H^{2-k}(\Omega))$ satisfying $\eqref{boundary}$. 

Moreover, if we assume \eqref{spd}, then there exists a constant $C>0$ independent of $u$ and $f$ such that
\begin{equation}\label{energy_estimate}E(t)\le C\|f\|_{L^2(\Omega)}^2\end{equation}
holds for all $t\in[0,T]$ and $\displaystyle u\in \bigcap_{k=1}^2H^k(0,T;H^{2-k}(\Omega))$ satisfying \eqref{boundary}.
\end{lem}

\begin{lem}\label{energy2}
Let $\ell\in\mathbb{N}$ be a fixed number, $A^0\in C^1(\overline{\Omega})$ satisfying $\displaystyle\min_{x\in\overline{\Omega}}A^0(x)>0$, $A\in C^1(\overline{\Omega};\mathbb{R}^d)$, $p\in W^{1,\infty}(0,T;L^\infty(\Omega))$, $R\in H^1(0,T;L^\infty(\Omega;\mathbb{R}^{\ell}))$, and $F\in L^2(\Omega;\mathbb{R}^\ell)$. Let us consider the initial boundary value problem
\begin{align}\label{boundary3}\begin{cases}A^0(x)\partial_tu+A(x)\cdot\nabla u+p(x,t)u=R(x,t)\cdot F(x)\quad &\text{in}\ Q,\\
u=0\quad &\text{on}\ \Gamma_{-,A}\times(0,T),\\
u(\cdot,0)=0\quad &\text{on}\ \Omega.\end{cases}\end{align} 
Then, there exists a constant $C>0$ independent of $u$ and $F$ such that
\begin{equation}\label{energy_estimate2}E(t)\le C\|F\|_{L^2(\Omega;\mathbb{R}^\ell))}^2\end{equation}
holds for all $t\in[0,T]$ and $\displaystyle u\in \bigcap_{k=1}^2H^k(0,T;H^{2-k}(\Omega))$ satisfying \eqref{boundary3}.
\end{lem}

\section{Proofs of main results}\label{Proof}
Using several estimates introduced in section \ref{Carleman_energy}, we prove the three main theorems in the subsequently sections.

\subsection{Proof of Theorem \ref{ISP}}
\begin{proof}[Proof of Theorem \ref{ISP}]
By our assumption \eqref{time}, we can take $0<\beta<\dfrac{\rho}{\displaystyle\sup_{x\in\Omega,t>0}A^0(x,t)}$ independent of $T$ satisfying
\[(T_0<)\frac{\displaystyle\max_{x\in\overline{\Omega}}\varphi_0(x)}{\beta}<T.\]
Then, there exists $\kappa>0$ such that
\begin{equation}\label{kappa}\max_{x\in\overline{\Omega}}\varphi_0(x)-\beta T<-\kappa.\end{equation}
Henceforth, by $C>0$ we denote a generic constant independent of $u$ which may change from line to line, unless specified otherwise. Applying the Carleman estimate \eqref{estimate} of Proposition \ref{Carleman} to $\displaystyle\partial_tu\in\bigcap_{k=0}^1H^k(0,T;H^{1-k}(\Omega))$ yields
\begin{align}\label{partial_t}
&s^2\int_{Q} e^{2s\varphi}|\partial_tu|^2dxdt+s\int_\Omega e^{2s\varphi(x,0)}|R(x,0)f(x)|^2dx\\
&\le C\int_{Q} e^{2s\varphi}|(P+p(x,t))\partial_tu|^2dxdt+Cs\int_{\Sigma_+} e^{2s\varphi}|\partial_tu|^2dSdt\notag\\
&\quad+Cs\int_\Omega e^{2s\varphi(x,T)}|\partial_tu(x,T)|^2dx\notag.\end{align}
Since we obtain
\begin{align*}(P+p(x,t))\partial_tu&=\partial_t\Big(A^0(x,t)\partial_tu+A(x,t)\cdot\nabla u+p(x,t)u\Big)\\
&\quad-\partial_tA^0(x,t)\partial_tu-\partial_tA(x,t)\cdot\nabla u-\partial_tp(x,t)u\\
&=\partial_tR(x,t)f(x)-\partial_tA^0(x,t)\partial_tu-\partial_tA(x,t)\cdot\nabla u-\partial_tp(x,t)u,\end{align*}
we have
\begin{align}\label{above}|(P+p(x,t))\partial_tu|^2&\le C\Big(|\partial_tRf|^2+|\partial_tu|^2+|\partial_tA(x,t)\cdot\nabla u|^2+|u|^2\Big)\\
&\le C\Big(|\partial_tRf|^2+|\partial_tu|^2+|A(x,t)\cdot\nabla u|^2+|u|^2\Big),\notag\end{align}
where we used the assumption \eqref{spd} to obtain the second inequality.
Therefore, applying the equation in \eqref{boundary} to the above estimate \eqref{above} yields
\begin{equation}\label{key}|(P+p(x,t))\partial_tu|^2\le C\Big(|\partial_tRf|^2+|Rf|^2+|\partial_tu|^2+|u|^2\Big).\end{equation}
Furthermore, applying \eqref{kappa} and the energy estimate \eqref{energy_estimate} of Lemma \ref{energy} yields
\begin{align}\label{T_1_estimate}s\int_\Omega e^{2s\varphi(x,T)}|\partial_tu(x,T)|^2dx&\le Cs e^{-2\kappa s}\int_\Omega A^0(x,T)|\partial_tu(x,T)|^2dx\\
&\le Cs e^{-2\kappa s}\|f\|_{L^2(\Omega)}^2\notag.\end{align}
Applying \eqref{key} and \eqref{T_1_estimate} to \eqref{partial_t} and choosing $s>s_*$ large enough yield
\begin{align}\label{partial_t'}
&s^2\int_{Q} e^{2s\varphi}|\partial_tu|^2dxdt+s\int_\Omega e^{2s\varphi(x,0)}|R(x,0)f(x)|^2dx\\
&\le C\int_{Q} e^{2s\varphi}\Big(\sum_{k=0}^1|\partial_t^kR|^2\Big)|f|^2dxdt+C\int_{Q}e^{2s\varphi}|u|^2dxdt\notag\\
&\quad+Cs\int_{\Sigma_+} e^{2s\varphi}|\partial_tu|^2dSdt+Cs e^{-2\kappa s}\|f\|_{L^2(\Omega)}^2.\notag\end{align}
In regard to the left-hand side of \eqref{partial_t'}, using \eqref{R}, for some $C>0$ we obtain
\begin{equation}\label{left}s^2\int_{Q} e^{2s\varphi}|\partial_tu|^2dxdt+s\int_\Omega e^{2s\varphi(x,0)}|R(x,0)f(x)|^2dx\ge Cs\|e^{s\varphi_0}f\|_{L^2(\Omega)}^2.\end{equation}
In regard to right-hand side of \eqref{partial_t'}, applying the Carleman estimate \eqref{estimate} of Proposition \ref{Carleman} to $\displaystyle u\in\bigcap_{k=1}^2H^{k}(0,T;H^{2-k}(\Omega))$ and then using \eqref{kappa} and the energy estimate \eqref{energy_estimate} yield
\begin{align}\label{right}
&\int_{Q} e^{2s\varphi}|u|^2dxdt\\
&\le \frac{C}{s^2}\int_{Q} e^{2s\varphi}|Rf|^2dxdt+\frac{C}{s}\int_{\Sigma_+} e^{2s\varphi}|u|^2dS dt\notag\\
&\quad+\frac{C}{s}\int_\Omega e^{2s\varphi(x,T)}|u(x,T)|^2dx\notag\\
&\le \frac{C}{s^2}\int_{Q} e^{2s\varphi}|Rf|^2dxdt+\frac{C}{s}\int_{\Sigma_+} e^{2s\varphi}|u|^2dS dt+\frac{C}{s}e^{-2\kappa s}\|f\|_{L^2(\Omega)}^2.\notag\end{align}
Applying \eqref{left} and \eqref{right} to \eqref{partial_t'} and choosing sufficiently large $s>s_*$ yield
\begin{align*}&s\|e^{s\varphi_0}f\|_{L^2(\Omega)}\\
&\le C\int_{Q} e^{2s\varphi}\Big(\sum_{k=0}^1|\partial_t^kR|^2\Big)|f|^2dxdt+\frac{C}{s}\int_{\Sigma_+}e^{2s\varphi}|u|^2dSdt\\
&\quad+Cs e^{Cs}\|\partial_tu\|_{L^2(\Sigma_+)}^2+Cs e^{-2\kappa s}\|f\|_{L^2(\Omega)}^2\\
&\le C\int_{Q} e^{2s\varphi}\Big(\sum_{k=0}^1|\partial_t^kR|^2\Big)|f|^2dxdt+Cse^{Cs}\sum_{k=0}^1\|\partial_t^ku\|_{L^2(\Sigma_+)}^2\\
&\quad+Cse^{-2\kappa s}\|f\|_{L^2(\Omega)}^2\\
&=C\int_\Omega\left(\int_0^{T} e^{-2s(\varphi_0(x)-\varphi(x,t))}\Big(\sum_{k=0}^1\|\partial_t^kR(\cdot,t)\|_{L^\infty(\Omega)}^2\Big)dt\right)e^{2s\varphi_0}|f|^2dx\\
&\quad+Cse^{Cs}\sum_{k=0}^1\|\partial_t^ku\|_{L^2(\Sigma_+)}^2+Cse^{-2\kappa s}\|f\|_{L^2(\Omega)}^2\\
&=C\int_\Omega\left(\int_0^{T} e^{-2\beta ts}\Big(\sum_{k=0}^1\|\partial_t^kR(\cdot,t)\|_{L^\infty(\Omega)}^2\Big)dt\right)e^{2s\varphi_0}|f|^2dx\\
&\quad+Cse^{Cs}\sum_{k=0}^1\|\partial_t^ku\|_{L^2(\Sigma_+)}^2+Cse^{-2\kappa s}\|f\|_{L^2(\Omega)}^2\\
&\le o(1)\|e^{s\varphi_0}f\|_{L^2(\Omega)}^2+Cs e^{Cs}\sum_{k=0}^1\|\partial_t^ku\|_{L^2(\Sigma_+)}^2+Cs e^{-2\kappa s}\|e^{s\varphi_0}f\|_{L^2(\Omega)}^2\\
&=o(1)\|e^{s\varphi_0}f\|_{L^2(\Omega)}^2+Cs e^{Cs}\sum_{k=0}^1\|\partial_t^ku\|_{L^2(\Sigma_+)}^2\end{align*}
as $s\to+\infty$ by the Lebesgue dominated convergence theorem. Choosing $s>s_*$ large enough yields
\[\|e^{s\varphi_0}f\|_{L^2(\Omega)}\le Ce^{Cs}\sum_{k=0}^1\|\partial_t^ku\|_{L^2(\Sigma_+)}.\]
Since $\varphi_0(x)\ge 0$ for all $x\in\overline{\Omega}$, $\|e^{s\varphi_0}f\|_{L^2(\Omega)}\ge \|f\|_{L^2(\Omega)}$ holds. Then, we complete the proof.
\end{proof}

\subsection{Proof of Theorem \ref{ICP}}\label{Proof_of_Theorem_ICP}
\begin{proof}[Proof of Theorem \ref{ICP}]
We show that Theorem \ref{ICP} comes down to Theorem \ref{ISP}. Setting
\[v:=u_1-u_2,\quad R:=-u_2,\quad f:=p_1-p_2,\]
we obtain
\begin{align*}\begin{cases}Pv+p_1(x)v=R(x,t)f(x)\quad &\text{in}\ Q,\\
v=0\quad &\text{on}\ \Sigma_-,\\
v(\cdot,0)=0\quad &\text{on}\ \Omega,\end{cases}\end{align*}
and \eqref{R} is satisfied due to the assumption \eqref{alpha}. Therefore, by Theorem \ref{ISP}, the proof is completed.
\end{proof}

\subsection{Proof of Theorem \ref{ICP2}}
\begin{proof}[Proof of Theorem \ref{ICP2}]
By our assumption \eqref{time2}, we can take $0<\beta<\dfrac{\rho}{\displaystyle\max_{x\in\overline{\Omega}}A^0_1(x)}$ independent of $T$ satisfying
\[\frac{\displaystyle\Big(\max_{x\in\overline{\Omega}}A^0_1(x)\Big)\Big(\max_{x\in\overline{\Omega}}\varphi_0(x)\Big)}{\rho}<\frac{\displaystyle\max_{x\in\overline{\Omega}}\varphi_0(x)}{\beta}<T.\]
Then, there exists $\kappa>0$ such that
\begin{equation}\label{kappa2}\max_{x\in\overline{\Omega}}\varphi_0(x)-\beta T<-\kappa.\end{equation}
Henceforth, by $C>0$ we denote a generic constant independent of $u$ which may change from line to line, unless specified otherwise. For $m=1,\ldots,d+1$, setting
\[v_m:=u_{1,m}-u_{2,m},\quad f_1:=A^0_1-A^0_2,\quad f_2:=A_1-A_2,\]
and
\begin{gather*}F:=\begin{pmatrix}f_1\\ f_2\end{pmatrix}\in L^2(\Omega;\mathbb{R}^{d+1}),\\
R_m:=\begin{pmatrix}-\partial_t u_{2,m}& -\partial_{x^1}u_{2,m}&\cdots &-\partial_{x^d}u_{2,m}\end{pmatrix}\in H^1(0,T;L^\infty(\Omega;\mathbb{R}^{d+1})).\end{gather*}
Thus, we obtain
\begin{align*}\begin{cases}P_1v_m+p(x,t)v_m=R_m(x,t)F(x)\quad &\text{in}\ Q,\\
v_m=0\quad &\text{on}\ \Sigma_-,\\
v_m(\cdot,0)=0\quad &\text{on}\ \Omega,\end{cases}\end{align*}
where the product in the right-hand side of the equation is a product of matrices. Applying the Carleman estimate \eqref{estimate} of Proposition \ref{Carleman} with $P=P_1$ to
\[\partial_tv_m\in\bigcap_{k=0}^1H^{k}(0,T;W^{1-k,\infty}(\Omega))\subset\bigcap_{k=0}^1H^{k}(0,T;H^{1-k}(\Omega))\]
yields
\begin{align*}
&s^2\int_{Q} e^{2s\varphi}|\partial_tv_m|^2dxdt+s\int_\Omega e^{2s\varphi(x,0)}|R_m(x,0)F(x)|^2dx\\
&\le C\int_{Q} e^{2s\varphi}|(P_1+p(x,t))\partial_tv_m|^2dxdt+Cs\int_{\Gamma_{+,A_1}\times(0,T)} e^{2s\varphi}|\partial_tv_m|^2dSdt\\
&\quad+Cs\int_\Omega e^{2s\varphi(x,T)}|\partial_tv_m(x,T)|^2dx.\end{align*}
Summing up with respect to $m=1,\ldots,d+1$ yields
\begin{align}\label{partial_t2}
&s^2\int_{Q} e^{2s\varphi}|\partial_tv|^2dxdt+s\int_\Omega e^{2s\varphi(x,0)}|R(x,0)F(x)|^2dx\\
&\le C\int_{Q} e^{2s\varphi}|(P_1+p(x,t))\partial_tv|^2dxdt+Cs\int_{\Gamma_{+,A_1}\times(0,T)} e^{2s\varphi}|\partial_tv|^2dSdt\notag\\
&\quad+Cs\int_\Omega e^{2s\varphi(x,T)}|\partial_tv(x,T)|^2dx\notag,\end{align}
where we define
\[v:=\begin{pmatrix}v_1\\ \vdots\\ v_{d+1}\end{pmatrix},\quad R:=\begin{pmatrix}R_1\\ \vdots\\ R_{d+1}\end{pmatrix},\quad (P_1+p(x,t))\partial_tv:=\begin{pmatrix}(P_1+p(x,t))\partial_tv_1\\ \vdots\\ (P_1+p(x,t))\partial_tv_{d+1}\end{pmatrix}.\]
Since we obtain
\begin{align*}(P_1+p(x,t))\partial_tv_m&=\partial_t\Big(A^0_1(x)\partial_tv_m+A_1(x)\cdot\nabla v_m+p(x,t)v_m\Big)-\partial_tp(x,t)v_m\\
&=\partial_t(R_mF)-\partial_tp(x,t)v_m\end{align*}
for each $m=1,\ldots,d+1$, we have
\begin{equation}\label{key2}|(P_1+p(x,t))\partial_tv|^2\le C\Big(|\partial_tRF|^2+|v|^2\Big).\end{equation}
Furthermore, applying \eqref{kappa2} and the energy estimate \eqref{energy_estimate2} of Lemma \ref{energy2} for $m=1,\ldots,d+1$ yields
\begin{align*}s\int_\Omega e^{2s\varphi(x,T)}|\partial_tv_m(x,T)|^2dx&\le Cs e^{-2\kappa s}\int_\Omega A^0_1(x,T)|\partial_tv_m(x,T)|^2dx\\
&\le Cs e^{-2\kappa s}\|F\|_{L^2(\Omega;\mathbb{R}^{d+1})}^2,\end{align*}
which implies
\begin{equation}\label{T_estimate}s\int_\Omega e^{2s\varphi(x,T)}|\partial_tv(x,T)|^2dx\le Cs e^{-2\kappa s}\|F\|_{L^2(\Omega;\mathbb{R}^{d+1})}^2.\end{equation}
Applying \eqref{key2} and \eqref{T_estimate} to \eqref{partial_t2} and choosing $s>s_*$ large enough yield
\begin{align}\label{partial_t'2}
&s^2\int_{Q} e^{2s\varphi}|\partial_tv|^2dxdt+s\int_\Omega e^{2s\varphi(x,0)}|R(x,0)F(x)|^2dx\\
&\le C\int_{Q} e^{2s\varphi}|\partial_tRF|^2dxdt+C\int_Qe^{2s\varphi}|v|^2dxdt\notag\\
&\quad+Cs\int_{\Gamma\times(0,T)} e^{2s\varphi}|\partial_tv|^2dSdt+Cs e^{-2\kappa s}\|F\|_{L^2(\Omega;\mathbb{R}^{d+1})}^2\notag.\end{align}
In regard to the left-hand side of \eqref{partial_t'2}, we obtain
\begin{gather}\label{left2}s^2\int_{Q} e^{2s\varphi}|\partial_tv|^2dxdt+s\int_\Omega e^{2s\varphi(x,0)}|R(x,0)F(x)|^2dx\\
\ge Cs\|e^{s\varphi_0}F\|_{L^2(\Omega;\mathbb{R}^{d+1})}^2\notag\end{gather}
for some $C>0$ by \eqref{R2}. Indeed, by $\displaystyle\min_{x\in\overline{\Omega}}A^0_2(x)\ge\rho>0$, it follows that
\begin{align*}|\det R(x,0)|&=\left|\det \begin{pmatrix} \partial_tu_{2,1}(x,0)& \cdots& \partial_tu_{2,d+1}(x,0)\\ \nabla u_{2,1}(x,0)& \cdots & \nabla u_{2,d+1}(x,0)\end{pmatrix}\right|\\
&\ge C\left|\det \begin{pmatrix} A_2\cdot\nabla\alpha_1+p(x,0)\alpha_1& \cdots& A_2\cdot\nabla\alpha_{d+1}+p(x,0)\alpha_{d+1}\\ \nabla\alpha_1& \cdots & \nabla\alpha_{d+1}\end{pmatrix}\right|\\
&= C\left|\det \begin{pmatrix} p(x,0)\alpha_1& \cdots& p(x,0)\alpha_{d+1}\\ \nabla\alpha_1& \cdots & \nabla\alpha_{d+1}\end{pmatrix}\right|\\
&= C|p(x,0)|\left|\det \begin{pmatrix} \alpha_1(x)& \cdots& \alpha_{d+1}(x)\\ \nabla\alpha_1(x)& \cdots & \nabla\alpha_{d+1}(x)\end{pmatrix}\right|\ge m_0\quad \text{a.e.}\ x\in\Omega.\end{align*}
In regard to the right-hand side of \eqref{partial_t'2}, applying the Carleman estimate \eqref{estimate} of Proposition \ref{Carleman} to $\displaystyle v_m\in\bigcap_{k=1}^2H^{k}(0,T;W^{2-k,\infty}(\Omega))$ for each $m=1,\ldots,d+1$ and then using \eqref{kappa2} and the energy estimate \eqref{energy_estimate2} of Lemma \ref{energy2} yield
\begin{align}\label{right2}
&\int_{Q} e^{2s\varphi}|v|^2dxdt\\
&\le \frac{C}{s^2}\int_{Q} e^{2s\varphi}|Rf|^2dxdt+\frac{C}{s}\int_{\Gamma_{+,A_1}\times(0,T)} e^{2s\varphi}|v|^2dS dt\notag\\
&\quad+\frac{C}{s}\int_\Omega e^{2s\varphi(x,T)}|v(x,T)|^2dx\notag\\
&\le \frac{C}{s^2}\int_{Q} e^{2s\varphi}|Rf|^2dxdt+\frac{C}{s}\int_{\Gamma\times(0,T)} e^{2s\varphi}|v|^2dS dt\notag\\
&\quad+\frac{C}{s}e^{-2\kappa s}\|F\|_{L^2(\Omega;\mathbb{R}^{d+1})}^2.\notag\end{align}
Applying \eqref{left2} and \eqref{right2} to \eqref{partial_t'2} and choosing sufficiently large $s>s_*$ yield
\begin{align*}&s\|e^{s\varphi_0}F\|_{L^2(\Omega;\mathbb{R}^{d+1})}^2\\
&\le C\int_{Q} e^{2s\varphi}|\partial_tRF|^2dxdt+\frac{C}{s^2}\int_Qe^{2s\varphi}|Rf|^2dxdt+\frac{C}{s}\int_{\Gamma\times(0,T)}e^{2s\varphi}|v|^2dSdt\\
&\quad+Cs\int_{\Gamma\times(0,T)} e^{2s\varphi}|\partial_tv|^2dSdt+Cs e^{-2\kappa s}\|F\|_{L^2(\Omega;\mathbb{R}^{d+1})}^2\\
&\le C\int_{Q} e^{2s\varphi}\Big(\sum_{k=0}^1|\partial_t^kRF|^2\Big)dxdt+Cse^{Cs}\|v\|_{H^1(0,T;L^2(\Gamma;\mathbb{R}^{d+1}))}\\
&\quad+Cs e^{-2\kappa s}\|F\|_{L^2(\Omega;\mathbb{R}^{d+1})}^2\\
&=C\int_\Omega\left(\int_0^T e^{-2\beta ts}\Big(\sum_{k=0}^1\|\partial_t^kR(\cdot,t)\|_{L^\infty(\Omega;\mathbb{R}^{(d+1)\times (d+1)})}^2\Big)dt\right)e^{2s\varphi_0}|F|^2dx\\
&\quad+Cse^{Cs}\|v\|_{H^1(0,T;L^2(\Gamma;\mathbb{R}^{d+1}))}^2+Cse^{-2\kappa s}\|F\|_{L^2(\Omega;\mathbb{R}^{d+1})}^2\\
&\le o(1)\|e^{s\varphi_0}F\|_{L^2(\Omega;\mathbb{R}^{d+1})}^2+Cse^{Cs}\|v\|_{H^1(0,T;L^2(\Gamma;\mathbb{R}^{d+1}))}^2+Cse^{-2\kappa s}\|e^{s\varphi_0}F\|_{L^2(\Omega;\mathbb{R}^{d+1})}^2\\
&=o(1)\|e^{s\varphi_0}F\|_{L^2(\Omega;\mathbb{R}^{d+1})}^2+Cse^{Cs}\|v\|_{H^1(0,T;L^2(\Gamma;\mathbb{R}^{d+1}))}^2\end{align*}
as $s\to+\infty$ by the Lebesgue dominated convergence theorem. Choosing $s>s_*$ large enough yields
\[\|e^{s\varphi_0}F\|_{L^2(\Omega;\mathbb{R}^{d+1})}^2\le Ce^{Cs}\|v\|_{H^1(0,T;L^2(\Gamma;\mathbb{R}^{d+1}))}^2\]
Since $\varphi_0(x)\ge 0$ for all $x\in\overline{\Omega}$, $\|e^{s\varphi_0}F\|_{L^2(\Omega;\mathbb{R}^{d+1})}^2\ge \|F\|_{L^2(\Omega;\mathbb{R}^{d+1})}^2$ holds. Then, we complete the proof.
\end{proof}

\section{Appendix}
In Appendix, we prove Proposition \ref{structure}, Lemma \ref{energy}, and Lemma \ref{energy2}.

\subsection{Proof of Proposition \ref{structure}}
\begin{proof}[Proof of Proposition \ref{structure}]
When $d\ge 2$, we note that there exists a vector-valued function $A_\perp(x,t)\neq 0$ for each $(x,t)\in\overline{Q}$ such that
\[A(x,t)\cdot A_\perp(x,t)=0\]
due to \eqref{positivity}. Applying \eqref{spd} to $\xi=A_\perp(x,t)$ yields
\[\forall(x,t)\in\overline{Q},\ \partial_tA(x,t)\cdot A_\perp(x,t)=0,\]
which implies that there exists a function $\phi\in C^1(\overline{Q})$ such that
\[\forall(x,t)\in\overline{Q},\quad \partial_tA(x,t)=\phi(x,t)A(x,t).\]
Therefore, $A(x,t)$ is represented by
\[A(x,t)=A(x,0)e^{\int_0^t\phi(x,s)ds}.\]
When $d=1$, noting Remark \ref{d=1}, setting
\[\phi(x,t):=\frac{\partial_tA(x,t)}{A(x,t)}\]
completes the proof.
\end{proof}

\subsection{Proof of Lemma \ref{energy}}

\begin{proof}[Proof of Lemma \ref{energy}]
Differentiating the equation in \eqref{boundary} with respect to $t$ yields
\begin{align*}A^0(x,t)\partial_t^2u+\partial_tA^0(x,t)\partial_tu+A(x,t)\cdot\nabla\partial_tu&\\
+\partial_tA(x,t)\cdot\nabla u+p(x,t)\partial_tu+\partial_tp(x,t)u&=\partial_tR(x,t)f(x).\end{align*}
Multiplying $2\partial_tu$ to the above equality and integrating over $\Omega$ yield
\begin{align*}&\int_\Omega A^0(x,t)\partial_t(|\partial_tu|^2)dx+\int_\Omega 2\partial_tA^0(x,t)|\partial_tu|^2dx+\int_\Omega A(x,t)\cdot\nabla(|\partial_tu|^2)dx\\
&\quad+\int_\Omega 2\partial_tu(\partial_tA(x,t)\cdot\nabla u)dx+\int_\Omega 2p(x,t)|\partial_tu|^2dx+\int_\Omega 2\partial_tp(x,t)u\partial_tudx\\
&=\int_\Omega 2\partial_tu \partial_tR(x,t)f(x)dx.\end{align*}
Integration by parts yields
\begin{align*}&\frac{d}{dt}\int_\Omega A^0(x,t)|\partial_tu|^2dx\\
&=-\int_\Omega (\partial_tA^0(x,t)+2p(x,t))|\partial_tu|^2dx+\int_\Omega(\nabla\cdot A(x,t))|\partial_tu|^2dx\\
&\quad-\int_\Omega 2\partial_tu(\partial_tA(x,t)\cdot\nabla u)dx-\int_\Omega 2\partial_tp(x,t)u\partial_tudx+\int_\Omega 2\partial_tu \partial_tRfdx\\
&\quad-\int_{\partial\Omega}(A(x,t)\cdot\nu)|\partial_tu|^2dS\\
&\le C\left(\int_\Omega A^0(x,t)|\partial_tu|^2dx+\int_\Omega |u|^2dx+\int_\Omega|\partial_tA(x,t)\cdot\nabla u|^2dx+\int_\Omega|\partial_tRf|^2dx\right)\\
&\quad-\int_{\partial\Omega}(A(x,t)\cdot\nu)|\partial_tu|^2dS.\end{align*}
Adding $\displaystyle\frac{d}{dt}\int_\Omega|u|^2dx$ to the both sides of the above estimate, we obtain
\begin{align}\label{Integration}&\frac{d}{dt}\left(\int_\Omega A^0(x,t)|\partial_tu|^2dx+\int_\Omega|u|^2dx\right)\\
&\le C\bigg(\int_\Omega A^0(x,t)|\partial_tu|^2dx+\int_\Omega |u|^2dx+\int_\Omega|\partial_tA(x,t)\cdot\nabla u|^2dx\notag\\
&\quad+\int_\Omega|\partial_tRf|^2dx\bigg)+\int_\Omega 2|u||\partial_tu|dx-\int_{\partial\Omega}(A(x,t)\cdot\nu)|\partial_tu|^2dS\notag\\
&\le C\bigg(\int_\Omega A^0(x,t)|\partial_tu|^2dx+\int_\Omega|u|^2dx+\int_\Omega|\partial_tA(x,t)\cdot\nabla u|^2dx\notag\\
&\quad+\int_\Omega|\partial_tRf|^2dx\bigg)-\int_{\partial\Omega}(A(x,t)\cdot\nu)|\partial_tu|^2dS,\notag\end{align}
which implies
\begin{align*}&\frac{d}{dt}\left(e^{-Ct}\int_\Omega \Big(A^0(x,t)|\partial_tu|^2+|u|^2\Big)dx\right)\\
&\le e^{-Ct}\left(C\int_\Omega\Big(|\partial_tA(x,t)\cdot\nabla u|^2+|\partial_tRf|^2\Big)dx-\int_{\partial\Omega}(A(x,t)\cdot\nu)|\partial_tu|^2dS\right).\end{align*}
Integrating over $(0,t)$ for $t\le T$ yields
\[E(t)\le C\left(E(0)+\int_{Q}|\partial_tA(x,t)\cdot\nabla u|^2dxdt+\int_\Omega|f|^2dx\right).\]
Since, using the equation \eqref{boundary}, we obtain
\begin{equation}\label{initial}E(0)\le C\int_\Omega|f|^2dx,\end{equation}
we prove \eqref{useless}.

Moreover, if we assume the assumption \eqref{spd}, then there exists $C>0$ such that for all $(x,t)\in \overline{Q}$,
\[|\partial_tA(x,t)\cdot\nabla u|^2\le C|A(x,t)\cdot\nabla u|^2.\]
Therefore, applying the above inequality to \eqref{Integration} and using the equation in \eqref{boundary} yield
\begin{align*}&\frac{d}{dt}\left(\int_\Omega A^0(x,t)|\partial_tu|^2dx+\int_\Omega |u|^2dx\right)\\
&\le C\left(\int_\Omega A^0(x,t)|\partial_tu|^2dx+\int_\Omega|u|^2dx+\int_\Omega|\partial_tRf|^2dx+\int_\Omega|Rf|^2dx\right)\\
&\quad-\int_{\partial\Omega}(A(x,t)\cdot\nu)|\partial_tu|^2dS,\end{align*}
which implies
\begin{align*}&\frac{d}{dt}\left(e^{-Ct}\int_\Omega \Big(A^0(x,t)|\partial_tu|^2+|u|^2\Big)dx\right)\\
&\le e^{-Ct}\left(C\int_\Omega\Big(\sum_{k=0}^1|\partial_t^kR|^2\Big)|f|^2dx-\int_{\partial\Omega}(A(x,t)\cdot\nu)|\partial_tu|^2dS\right).\end{align*}
Integrating over $(0,t)$ for $t\le T$ yields
\[E(t)\le C\left(E(0)+\int_\Omega|f|^2dx\right).\]
By \eqref{initial}, we complete the proof.
\end{proof}

\subsection{Proof of Lemma \ref{energy2}}

\begin{proof}[Proof of Lemma \ref{energy2}]
Differentiating the equation with respect to $t$ yields
\[A^0(x)\partial_t^2u+A(x)\cdot\nabla\partial_tu+p(x,t)\partial_tu+\partial_tp(x,t)u=\partial_tR(x,t)\cdot F(x).\]
Multiplying $2\partial_tu$ to the above equation and integrating over $\Omega$ yield
\begin{align*}\int_\Omega A^0(x)\partial_t(|\partial_tu|^2)dx+\int_\Omega A(x)\cdot\nabla(|\partial_tu|^2)dx&\\
+\int_\Omega 2p(x,t)|\partial_tu|^2dx+\int_\Omega 2\partial_tp(x,t)u\partial_tudx&=\int_\Omega 2\partial_tu \partial_tR(x,t)\cdot F(x)dx.\end{align*}
Integration by parts yields
\begin{align*}&\frac{d}{dt}\int_\Omega A^0(x)|\partial_tu|^2dx\\
&=\int_\Omega(\nabla\cdot A(x)-2p(x,t))|\partial_tu|^2dx-\int_\Omega 2\partial_tp(x,t)u\partial_tudx+\int_\Omega 2\partial_tu \partial_tR\cdot Fdx\\
&\quad-\int_{\partial\Omega}(A(x)\cdot\nu)|\partial_tu|^2dS\\
&\le C\left(\int_\Omega A^0(x)|\partial_tu|^2dx+\int_\Omega |u|^2dx+\int_\Omega|\partial_tR\cdot F|^2dx\right)-\int_{\partial\Omega}(A(x)\cdot\nu)|\partial_tu|^2dS.\end{align*}
Adding $\displaystyle\frac{d}{dt}\int_\Omega|u|^2dx$ to the both sides of the above estimate, we obtain
\begin{align*}&\frac{d}{dt}\left(\int_\Omega A^0(x)|\partial_tu|^2dx+\int_\Omega|u|^2dx\right)\\
&\le C\left(\int_\Omega A^0(x)|\partial_tu|^2dx+\int_\Omega |u|^2dx+\int_\Omega|\partial_tR\cdot F|^2dx\right)+\int_\Omega 2|u||\partial_tu|dx\\
&\quad-\int_{\partial\Omega}(A(x)\cdot\nu)|\partial_tu|^2dS\\
&\le C\left(\int_\Omega A^0(x)|\partial_tu|^2dx+\int_\Omega|u|^2dx+\int_\Omega|\partial_tR\cdot F|^2dx\right)-\int_{\partial\Omega}(A(x)\cdot\nu)|\partial_tu|^2dS,\end{align*}
which implies
\begin{align*}&\frac{d}{dt}\left(e^{-Ct}\int_\Omega \Big(A^0(x)|\partial_tu|^2+|u|^2\Big)dx\right)\\
&\le e^{-Ct}\left(C\int_\Omega|\partial_tR\cdot F|^2dx-\int_{\partial\Omega}(A(x)\cdot\nu)|\partial_tu|^2dS\right).\end{align*}
Integrating over $(0,t)$ for $t\le T$ yields
\[E(t)\le C\left(E(0)+\int_\Omega |F|^2dx\right).\]
Since, using the equation in \eqref{boundary3}, we obtain
\[E(0)\le C\int_{\Omega}|F|^2dx,\]
we prove \eqref{energy_estimate2}.
\end{proof}

\section*{Acknowledgment}
This work was supported in part by Grant-in-Aid for JSPS Fellows Grant Number JP20J11497, and Istituto Nazionale di Alta Matematica (IN$\delta$AM), through the GNAMPA Research Project 2020. Istituto Nazionale di Alta Matematica (IN$\delta$AM), through the GNAMPA Research Project 2020, titled ``Problemi inversi e di controllo per equazioni di evoluzione e loro applicazioni'', coordinated by the first author. Moreover, this research was performed in the framework of the French-German-Italian Laboratoire International Associ\'e (LIA), named COPDESC, on Applied Analysis, issued by CNRS, MPI, and IN$\delta$AM, during the IN$\delta$AM Intensive Period-2019, ``{\it Shape optimization, control and inverse problems for PDEs}'', held in Napoli in May-June-July 2019.

This paper owes much to the thoughtful and helpful comments of Professor Piermarco Cannarsa (University of Rome ``Tor Vergata'') and Professor Masahiro Yamamoto (The University of Tokyo). In particular, we thank Prof. Cannarsa to have suggested to us that the inequality \eqref{spd} imply a precise exponential structure for the coefficient $A(x,t)$ (that we showed in Proposition \ref{structure}). Moreover, we are extremely grateful to Prof. Yamamoto to have fully read a draft of this paper and to have given us a lot of pieces of advice to make it more clear and readable.

\bibliographystyle{plain}
\bibliography{first_hyperbolic_r1}

\end{document}